\documentclass[12pt,reqno]{amsart}
\usepackage{amssymb}
\usepackage{hyperref}

\newtheorem{theorem}{Theorem}[section]
\newtheorem{lemma}[theorem]{Lemma}
\newtheorem{proposition}[theorem]{Proposition}
\newtheorem{corollary}[theorem]{Corollary}

\theoremstyle{definition}
\newtheorem{definition}[theorem]{Definition}

\numberwithin{equation}{section}

\begin{document}

\baselineskip=15.5pt

\title[Homogeneous Higgs and co-Higgs bundles]{Homogeneous Higgs and co-Higgs bundles on 
Hermitian symmetric spaces}

\author[I. Biswas]{Indranil Biswas}

\address{School of Mathematics, Tata Institute of Fundamental Research,
Homi Bhabha Road, Mumbai 400005}

\email{indranil@math.tifr.res.in}

\author[S. Rayan]{Steven Rayan}

\address{Centre for Quantum Topology and Its Applications (quanTA) and Department of Mathematics \& 
Statistics, University of Saskatchewan, McLean Hall, 106 Wiggins Road, Saskatoon, SK, S7N 5E6, 
Canada}

\email{rayan@math.usask.ca}



\date{}

\begin{abstract}
We define homogeneous principal Higgs and co-Higgs bundles over irreducible Hermitian 
symmetric spaces of compact type. We provide a classification for each type of object up to 
isomorphism, which in each case can be interpreted as defining a moduli space.
\end{abstract}

\maketitle

\tableofcontents

\section{Introduction}

In contrast to parabolic Higgs bundles and other types of meromorphic Higgs bundles, 
\emph{co-Higgs bundles} are an extension of the theory of Higgs bundles to the Fano end of the 
Kodaira spectrum that does not require considering open varieties. As in \cite{R0}, a co-Higgs 
bundle on a complex variety $X$ is a holomorphic bundle $E$ with a \emph{co-Higgs field} 
$\theta:E\to E\otimes T^{1,0}X$ such that the section $\theta\wedge\theta
\,\in\, H^0(X,\, \mbox{End}(E)\otimes\wedge^2T^{1,0}X)$ vanishes identically. These objects bear
a formal similarity to the 
Higgs bundles of \cite{H1,H2,Si}. For a Higgs bundle, there is a \emph{Higgs field} $\theta\,:\,E\,
\longrightarrow\, E\otimes\Omega^{1,0}X$ with $\theta\wedge\theta\,=\,0$.
The symmetry condition is required in 
higher dimensions and is the analogue of the integrability condition introduced in \cite{Si}.
Higgs bundles have been investigated over the past thirty years in what is by now a large body 
of work. The interest in co-Higgs bundles is more recent. While Higgs 
bundles originate from considerations of invariance and self-duality in $4$D Yang-Mills gauge 
theory \cite{H1}, co-Higgs bundles arose out of generalized complex geometry \cite{Gu,H3} as a 
kind of limit of generalized holomorphic bundles when generalized complex structures become 
ordinary complex ones. Co-Higgs bundles and their moduli have been investigated on $\mathbb 
P^1$ \cite{R1,R3,R4,BG3}; on $\mathbb P^2$ \cite{R2}; on $\mathbb P^1\times\mathbb P^1$ \cite{VC};
on the moduli space of stable vector
bundles on a complex curve of genus at least $2$ \cite{BR}, which is a compact 
Fano variety (motivated by constructions in \cite{FW}); on Calabi-Yau manifolds \cite{BB};
and in various settings and levels of generality by \cite{Py,HM,BH1,BH2,Co1,Co2}.

In the present work, we examine both Higgs and co-Higgs bundles on compact Hermitian symmetric 
spaces $M$. Investigations in this direction for Higgs bundles had commenced in \cite{BG1,BG2}. To 
continue this, we take advantage of the identification of each such $M$ with the quotient $G/K$ 
where $G$ is the simply-connected covering of the group of holomorphic isometries of $M$ and $K$ 
is the isotropy subgroup of a fixed base point $x_0\,\in\, M$. We can recover $M$ as the orbit of 
the base point $x_0\cong K/K$ under the canonical action of $G$. We define natural notions of 
Higgs and co-Higgs bundles adapted to such a space. These are the \emph{homogeneous Higgs 
bundles} (respectively, \emph{homogeneous co-Higgs bundles}) in which the underlying bundle is
a homogeneous holomorphic principal $H$-bundle $E_H$, where $H$ 
is a connected complex group, and where $\theta$ is regarded as a global $1$-form-twisted 
section (respectively, vector-field-twisted section) of the adjoint bundle of $E_H$ that is 
furthermore invariant under the action of $G$. We proceed to classify $G$-holomorphic 
structures on smooth homogeneous principal $H$-bundles $E_H\,\longrightarrow\, M\,\cong\, G/K$.
For us, \emph{$G$-holomorphic} means that the self-map of the total space of $E_H$ induced by
the action of each $g\, \in\, G$ is holomorphic. We refer to these simply as \emph{invariant} holomorphic 
structures. In particular, we prove that every homogeneous bundle $E_H$ has a canonical invariant holomorphic 
structure and that this structure arises from a particular invariant connection on $E_H$. This 
is Theorem \ref{thm0}. We subsequently provide a Lie-algebraic characterization (in terms of 
the vanishing of a natural tri-linear map on $\mathfrak h\,=\,\mbox{Lie}(H)$) for when an invariant 
almost complex structure on the bundle is integrable (see Proposition \ref{prop1}). Taken 
together, these statements results in a completely Lie-theoretic characterization of isomorphism 
classes of (based) homogeneous, invariant holomorphic principal-$H$ bundles in Theorem 
\ref{thm1}.

This analysis of invariant connections and integrable holomorphic structures on smooth 
bundles $E_H$ enables us to classify homogeneous principal (co-)Higgs bundles on $G/K$ up to 
isomorphism using exclusively Lie-theoretic data (see Theorem \ref{thm2}, Corollary 
\ref{cor2}, Theorem \ref{prop2} and Corollary \ref{cor3}). The resulting descriptions 
can be interpreted as equations for defining algebraic moduli varieties within linear spaces 
arising directly from $G$ and $H$.

\section{Homogeneous bundles on Hermitian symmetric spaces}\label{sec2}

Let $M$ be an irreducible Hermitian symmetric space of compact type. Fix a base point $x_0\, \in\, M$.
Let $G$ denote the simply-connected covering of the group of holomorphic isometries of $M$. Therefore,
$G$ has a tautological action on $M$. Let $K\,<\,G$ be the isotropy of the point $x_0$. Consequently,
we have $M\,=\, G/K$, and the point $x_0$ corresponds to $K/K\, \in\, G/K$. We shall identify $(M,\,
x_0)$ with $(G/K,\, K/K)$ by considering $M$ as the orbit of $x_0$ under the tautological action of
$G$ on $M$. This identification between $(M,\,x_0)$ and $(G/K,\, K/K)$ takes the action of any
$g\,\in\, G$ on $M$ to the left--translation action
\begin{equation}\label{wtg}
t_g\, :\, G/K\, \longrightarrow\, G/K\, ,\ \ g'K\, \longmapsto\, gg'K\,.
\end{equation}
Note that $t_g$ is an holomorphic isometry.

The quotient map
\begin{equation}\label{f1}
q\,:\, G\, \longrightarrow\, G/K
\end{equation}
defines a $C^\infty$ principal $K$--bundle over $G/K$. This principal $K$--bundle over
$G/K$ will be denoted by $\mathbb G$.

\begin{definition}\label{def-2}
Let $\mathcal H$ be a Lie group.
A \textit{homogeneous} $C^\infty$ principal $\mathcal H$--bundle
over $G/K$ is a pair of the form $(E_{\mathcal H},\, \rho)$, where
$f\, :\, E_{\mathcal H}\, \longrightarrow\, G/K$ is a $C^\infty$ principal
$\mathcal H$--bundle and
$$
\rho\, :\, G\times E_{\mathcal H}\, \longrightarrow\, E_{\mathcal H}
$$
is a $C^\infty$ left--action of the group $G$ on $E_{\mathcal H}$ satisfying the
following two conditions:
\begin{enumerate}
\item $f(\rho (g,\, z))\, =\, t_g(f(z))$, for all $g\, \in\, G$,
$z\, \in\, E_{\mathcal H}$, where $t_g$ is
the automorphism of $G/K$ in \eqref{wtg}, and

\item the actions of $G$ and $\mathcal H$ on $E_{\mathcal H}$ commute.
\end{enumerate}
\end{definition}

Therefore, a homogeneous $C^\infty$ principal $\mathcal H$--bundle $E_{\mathcal H}$
is equipped with an action of $G\times \mathcal H$, with $G$ acting on the left of
$E_{\mathcal H}$ and ${\mathcal H}$ acting on its right.

An \textit{isomorphism} between two homogeneous $\mathcal H$--bundles
$$
f\, :\, E_{\mathcal H}\, \longrightarrow\, G/K\ \ \text{ and } \ \
f'\, :\, E'_{\mathcal H}\, \longrightarrow\, G/K
$$
is a diffeomorphism
$$
\delta\, :\, E_{\mathcal H}\, \longrightarrow\,E'_{\mathcal H}
$$
that satisfies the following two conditions:
\begin{itemize}
\item $\delta$ is $G\times \mathcal H$--equivariant for the
actions of $G\times \mathcal H$ on $E_{\mathcal H}$
and $E'_{\mathcal H}$, and

\item $f'\circ\delta\,=\, f$.
\end{itemize}

Two homogeneous $C^\infty$ principal bundles are called \textit{isomorphic} if there is
an isomorphism between them.

A {\it based} homogeneous $C^\infty$ principal $\mathcal H$--bundle
over $G/K$ is a homogeneous $C^\infty$ principal $\mathcal H$--bundle
$(E_{\mathcal H},\, \rho)$ over $G/K$ together with a point
$z\,\in\, (E_{\mathcal H})_{K/K}$ in the fiber of $E_{\mathcal H}$ over
the point $K/K$. An \textit{isomorphism}
between two based homogeneous $C^\infty$ principal $\mathcal H$--bundles
$(E_{\mathcal H},\, \rho,\, z)$ and $(E'_{\mathcal H},\, \rho',\, z')$ is
an isomorphism of homogeneous $C^\infty$ principal bundles 
$$
\delta\, :\, E_{\mathcal H}\, \longrightarrow\,E'_{\mathcal H}
$$
such that $\delta(z)\,=\, z'$.

Note that the left--translation action of $G$ on itself makes the
principal $K$--bundle $\mathbb G$ in \eqref{f1} a 
homogeneous principal $K$--bundle. The identity element $e$ of $G\,=\,\mathbb G$
makes it a based homogeneous principal $K$--bundle.

\begin{lemma}\label{lem1}
The based homogeneous principal $\mathcal H$--bundles
on $G/K$ are in bijection with the
homomorphisms from $K$ to $\mathcal H$.
\end{lemma}

\begin{proof}
First take a homomorphism $\eta\, :\, K\, \longrightarrow\, {\mathcal H}$.
Consider the principal $K$--bundle $\mathbb G$ in \eqref{f1}. Let
${\mathbb G}_\eta$ be the principal $\mathcal H$--bundle on $G/K$ obtained
by extending the structure group of $\mathbb G$ using $\eta$. So,
${\mathbb G}_\eta$ is the quotient of $G\times \mathcal H$ where
two elements $(g_1, \, h_1),\, (g_2, \, h_2)\,\in\, G\times \mathcal H$ are
identified if there is an element $k\, \in\, K$ such that
$g_2\,=\, g_1k$ and $h_2\,=\, \eta(k)^{-1}h_1$. Consider the action of $G$ on
$G\times \mathcal H$ given by the left--translation action
of $G$ on itself and the trivial action of $G$ on $\mathcal H$. This action
of $G$ on $G\times \mathcal H$ produces an action
of $G$ on ${\mathbb G}_\eta$. The resulting action of $G$ on ${\mathbb G}_\eta$
makes ${\mathbb G}_\eta$ a homogeneous principal $\mathcal H$--bundle
on $G/K$. The image in ${\mathbb G}_\eta$ of the point $e\times e_{\mathcal H}\,\in\,
G\times {\mathcal H}$, where $e_{\mathcal H}$ is the identity element of $\mathcal H$,
is a point in the fiber of ${\mathbb G}_\eta$ over $K/K$, so $({\mathbb G}_\eta,\,
e\times e_{\mathcal H})$ is a based homogeneous principal $\mathcal H$--bundle.

For the converse, take a homogeneous principal $\mathcal H$--bundle
$(E_{\mathcal H},\, \rho)$ on $G/K$ together with a point $z_0$ in the
fiber of $E_{\mathcal H}$ over the point $K/K\, \in\, G/K$.
For any element $k\, \in\, K$, let $\eta(k)\, \in\, \mathcal H$ be the unique
element that satisfies the equation
\begin{equation}\label{se}
\rho(k,\, z_0)\,=\, z_0\eta(k)\, ;
\end{equation}
note that since $\rho(k,\, z_0)$ lies in the fiber of $E_{\mathcal H}$ over the
point $K/K\, \in\, G/K$, there is a unique such $\eta(k)$. Now, for $k,\, k'\,\in\, K$, we have
$$
z_0\eta(kk')\,=\, \rho(kk',\, z_0)\,=\, \rho(k,\, \rho(k',\, z_0))\,=\,
\rho(k,\, z_0\eta(k'))
$$
$$
=\, \rho(k,\, z_0)\eta(k')\,=\, z_0\eta(k)\eta(k')\, .
$$
This implies that the map
\begin{equation}\label{eta}
\eta\, :\, K\, \longrightarrow\, {\mathcal H}\, ,\ \ k\, \longmapsto\, \eta(k)\, ,
\end{equation}
is a homomorphism of groups.

For the above homomorphism $\eta$ consider the based principal $\mathcal H$--bundle
$({\mathbb G}_\eta,\, e\times e_{\mathcal H})$ constructed earlier from $\eta$. We shall show
that $({\mathbb G}_\eta,\, e\times e_{\mathcal H})$ is identified with
$(E_{\mathcal H},\, z_0)$. For this consider the map
$$
\eta'\, :\, G\times {\mathcal H}\, \longrightarrow\, E_{\mathcal H}
$$
that sends any $(g,\, h)\, \in\, G\times {\mathcal H}$ to $\rho(g,\, z_0h)$. It can be
shown that $\eta'$ descends to a map
\begin{equation}\label{dm}
\eta''\, :\, {\mathbb G}_\eta\, \longrightarrow\, E_{\mathcal H}
\end{equation}
from the quotient ${\mathbb G}_\eta$ of $G\times {\mathcal H}$. Indeed,
for any $(k,\, g,\, h)\, \in\, K\times G\times {\mathcal H}$, we have
$$
\eta'(gk,\, \eta(k)^{-1}h)\,=\, \rho(gk,\, z_0\eta(k)^{-1}h) \,=\, \rho(g,\,
\rho(k,\, z_0))\eta(k)^{-1}h
$$
$$
=\, \rho(g,\, z_0\eta(k))\eta(k)^{-1}h \,=\,
\rho(g,\, z_0)\eta(k)\eta(k)^{-1}h\,=\, \rho(g,\, z_0)h\,=\, \rho(g,\, z_0h)\,=\, \eta'(g,\, h)\, .
$$
Therefore, $\eta'$ descends to a map $\eta''$ as in \eqref{dm}.
This map $\eta''$ in \eqref{dm} is an isomorphism of homogeneous principal $\mathcal H$--bundles.
Note that $\eta''$ clearly sends $e\times e_{\mathcal H}$ to $z_0$.

Conversely, take a homomorphism $\eta\, :\, K\, \longrightarrow\, {\mathcal H}$. Let 
$({\mathbb G}_\eta,\,e\times e_{\mathcal H})$ be the based homogeneous principal $\mathcal 
H$--bundle constructed as above using it. Then the homomorphism $K\, \longrightarrow\, {\mathcal H}$
constructed as in \eqref{se} for this based homogeneous principal $\mathcal H$--bundle clearly
coincides with $\eta$.

Consequently, the above two constructions, between
the space of homomorphisms from $K$ to $\mathcal H$ and the space of based homogeneous
$\mathcal H$--bundles, are inverses of each other.
\end{proof}

\section{A tautological connection}\label{se2.2}

A connection on a principal ${\mathcal H}$--bundle $f\, :\, E_{\mathcal H}\, \longrightarrow\,
G/K$ is a $C^\infty$ ${\mathcal H}$--invariant distribution ${\mathbb D}\, \subset\, TE_{\mathcal H}$
such that the natural homomorphism ${\mathbb D}\oplus \text{kernel}(df)\, \longrightarrow\,
TE_{\mathcal H}$ is an isomorphism, where $df\, :\, TE_{\mathcal H}\, \longrightarrow\, f^*T(G/K)$
is the differential of $f$ \cite{At}, \cite{KN}.

Take a homogeneous principal $\mathcal H$--bundle
$(E_{\mathcal H},\, \rho)$ on $G/K$.
For any $g\, \in\, G$, let $\rho_g$ be the diffeomorphism of $E_{\mathcal H}$
defined by $z\, \longmapsto\, \rho(g,\, z)$. This $\rho_g$ is $\mathcal H$--equivariant;
more precisely, it is an automorphism
of the principal $\mathcal H$--bundle $E_{\mathcal H}$ over the biholomorphism
$t_g$ of $G/K$ defined in \eqref{wtg}. Let $C(E_{\mathcal H})$ denote the
space of all connections on the principal $\mathcal H$--bundle $E_{\mathcal H}$. The group
$G$ acts on $C(E_{\mathcal H})$ as follows: the action of any $g\, \in\, G$ sends
the connection defined by a distribution ${\mathbb D}\, \subset\, TE_{\mathcal H}$
to the connection $$d\rho_g({\mathbb D})\, \subset\, TE_{\mathcal H}\, ,$$ where
$d\rho_g\, :\, TE_{\mathcal H}\, \longrightarrow\, TE_{\mathcal H}$ is the differential
of the map $\rho_g$. Let
\begin{equation}\label{f2}
C(E_{\mathcal H})^G\,\subset\, C(E_{\mathcal H})
\end{equation}
be the fixed point locus for this action of $G$ on $C(E_{\mathcal H})$.

Consider the homogeneous principal $K$--bundle $\mathbb G$ in \eqref{f1}.
We shall show that it has a tautological $G$--invariant connection.

Let $\mathfrak g$ (respectively, $\mathfrak k$) be the Lie algebra of
$G$ (respectively, $K$). Both $\mathfrak g$ and $\mathfrak k$ are $K$--modules
by the adjoint action. The Killing form on $\mathfrak g$ is non-degenerate.
Let
$$
{\mathfrak p}\, :=\, {\mathfrak k}^\perp \,\subset\, \mathfrak g
$$
be the orthogonal complement of $\mathfrak k$ for the Killing form on $\mathfrak g$.
The adjoint action 
of $K$ on $\mathfrak g$ preserves $\mathfrak p$, because the Killing form on $\mathfrak g$
is $K$--invariant. Therefore, the natural
homomorphism
\begin{equation}\label{f3}
{\mathfrak k}\oplus {\mathfrak p}\,\longrightarrow\, \mathfrak g
\end{equation}
is an isomorphism of $K$--modules.

Now the translations of ${\mathfrak p}$ by the
left--translation action of $G$ on itself define a distribution
$$
D\, \subset\, TG\, .
$$
This $D$ is preserved by the right--translation action of $K$ on $G$ because
the decomposition in \eqref{f3} is an isomorphism of $K$--modules. From this it follows that
$D$ defines a connection on the principal $K$--bundle $\mathbb G$ in \eqref{f1}. This
connection on $\mathbb G$ will be denoted by $\nabla^0$. Since $D$ is preserved
by the left--translation action of $G$ on itself, we conclude that
\begin{equation}\label{f4}
\nabla^0\, \in\, C({\mathbb G})^G
\end{equation}
(see \eqref{f2}).

Consider the Lie bracket operation composed with the projection to the direct summand $\mathfrak k$
in \eqref{f3}
\begin{equation}\label{g1}
{\mathfrak p}\otimes {\mathfrak p}\, \longrightarrow\, {\mathfrak g}
\, \longrightarrow\, {\mathfrak k}\, .
\end{equation}
The tangent bundle $T(G/K)$ is the vector bundle over $G/K$
associated to the principal $K$--bundle $\mathbb G$
in \eqref{f1} for the adjoint action of $K$ on $\mathfrak p$. Therefore, the composition
homomorphism in \eqref{g1}, which is $K$--equivariant, define a
$C^\infty$ two--form on $G/K$ with values in the adjoint vector bundle $\text{ad}({\mathbb G})$.
This $\text{ad}({\mathbb G})$--valued two--form on $G/K$ is the curvature of
the above connection $\nabla^0$. We shall denote the
curvature of $\nabla^0$ by ${\mathcal K}(\nabla^0)$.

The center of $K$ will be denoted by $Z_K$; it is isomorphic to ${\rm U}(1)$, because
the Hermitian symmetric space $G/K$ is irreducible. Consider the action of $Z_K$ on
the complexification of $\mathfrak p$. Let
\begin{equation}\label{d1}
{\mathfrak p}^{\mathbb C}\, :=\,
{\mathfrak p}\otimes_{\mathbb R} {\mathbb C}\,=\, {\mathfrak p}_+\oplus {\mathfrak p}_{-}
\end{equation}
be the isotypical decomposition for the action of $Z_K$ on ${\mathfrak p}^{\mathbb C}$. Note
that $${\mathfrak p}^{\mathbb C}\,=\, T_{K/K}(G/K)\otimes_{\mathbb R} {\mathbb C}\, .$$
The type decomposition given by the complex structure on $G/K$
$$
T_{K/K}(G/K)\otimes_{\mathbb R} {\mathbb C}\,=\, T^{1,0}_{K/K}(G/K)\oplus T^{0,1}_{K/K}(G/K)
$$
coincides with the decomposition in \eqref{d1}; the complex subspace
${\mathfrak p}_+$ (respectively, ${\mathfrak p}_{-}$) of ${\mathfrak p}\otimes_{\mathbb R}
{\mathbb C}$ coincides with $T^{1,0}_{K/K}(G/K)$ (respectively, $T^{0,1}_{K/K}(G/K)$) by the
above isomorphism ${\mathfrak p}^{\mathbb C}\,=\, T_{K/K}(G/K)\otimes_{\mathbb R} {\mathbb C}$.
(See \cite{He} for the details.)

The complexification of the composition in \eqref{g1} vanishes on ${\mathfrak p}_+
\otimes {\mathfrak p}_+$ and ${\mathfrak p}_{-}\otimes {\mathfrak p}_{-}$. From this
it follows immediately that both the
$(2,\, 0)$ and $(0,\, 2)$ type components of the curvature ${\mathcal K}(\nabla^0)(K/K)$
of $\nabla^0$ vanish (at the point $K/K\, \in\, G/K$).

Note that from the fact that the connection $\nabla^0$ is $G$--invariant (see \eqref{f4}) it 
follows immediately that the curvature ${\mathcal K}(\nabla^0)$ is preserved by the action of 
$G$. Hence ${\mathcal K}(\nabla^0)$ is an $\text{ad}({\mathbb G})$--valued form of
of Hodge type $(1, \,1)$ on $G/K$, because it is of type 
$(1, \,1)$ at the point $K/K\, \in\, G/K$ and the action of $G$ on $G/K$ is transitive.

\section{Invariant holomorphic structures}

Let $H$ be a connected complex Lie group.
The Lie algebra of $H$ will be denoted by $\mathfrak h$.

A holomorphic structure on a $C^\infty$ principal $H$--bundle
\begin{equation}\label{prf}
f\, :\, E_H\, \longrightarrow\, G/K
\end{equation}
is a complex structure on the manifold $E_H$ such that the projection $f$ is holomorphic and 
the map $E_H\times H\, \longrightarrow\, E_H$ giving the action of $H$ on $E_H$ is holomorphic. A 
holomorphic principal $H$--bundle is a $C^\infty$ principal $H$--bundle equipped with a holomorphic 
structure.

Now let $E_H$ be homogeneous. A holomorphic structure on $E_H$ is called {\it invariant} if for
every $g\, \in\, G$ the self-map of $E_H$ given by the action of $g$ on it is holomorphic.

\begin{theorem}\label{thm0}
Any $C^\infty$ homogeneous principal $H$--bundle $E_H\, \longrightarrow\, G/K$ has a tautological
invariant connection $\nabla^{E_H}\, \in\, C(E_H)^G$ (defined in \eqref{f2}). This connection $\nabla^{E_H}$
produces an invariant holomorphic structure on $E_H$.
\end{theorem}

\begin{proof}
The action of $G$ on $E_H$ produces a Lie algebra homomorphism
$$
\delta\, :\, \text{Lie}(G)\,=:\, {\mathfrak g}\, \longrightarrow\, C^\infty(E_H,\, TE_H)
$$
to the $C^\infty$ vector fields on $E_H$. Let
\begin{equation}\label{cd}
{\mathfrak D}\, :=\, \delta({\mathfrak p})\, \subset\, TE_H
\end{equation}
be the distribution defined by the image of the subspace ${\mathfrak p}$ in 
\eqref{f3}. Since the actions of $G$ and $H$ on $E_H$ commute, for any $v\, \in\, \mathfrak g$,
the above defined
vector field $\delta(v)$ on $E_H$ is preserved by the action of $H$ on $E_H$. Consequently,
the distribution ${\mathfrak D}$ in \eqref{cd} is also
preserved by the action of $H$ on $E_H$. Clearly, ${\mathfrak D}$ is transversal to the fibers 
of the projection $f$ in \eqref{prf}. Hence ${\mathfrak D}$ defines a connection on
the principal $H$--bundle $E_H$; this 
connection on $E_H$ will be denoted by $\nabla^{E_H}$. We have
$$\nabla^{E_H}\, \in\, C(E_H)^G\, ,$$ because 
the distribution ${\mathfrak D}$ is preserved by the action of $G$ on $E_H$.

Consider the connection $\nabla^0$ on the principal $K$--bundle ${\mathbb G}$ constructed in 
\eqref{f4}. From the proof of Lemma \ref{lem1} we know that the principal $H$--bundle $E_H$ is 
the extension of the structure group of the principal $K$--bundle ${\mathbb G}$ in \eqref{f1} 
using a homomorphism $K\, \longrightarrow\, H$. Consequently, the connection $\nabla^0$ on 
$\mathbb G$ induces a connection on $E_H$. From the construction of the above connection 
$\nabla^{E_H}$ on $E_H$ it is evident that $\nabla^{E_H}$ coincides with the connection on $E_H$ 
induced by $\nabla^0$.

It was shown in Section \ref{se2.2} that the curvature ${\mathcal K}(\nabla^0)$ of $\nabla^0$ is 
of type $(1,\, 1)$. This implies that the curvature of the induced connection $\nabla^{E_H}$ on 
$E_H$ is also of type $(1,\, 1)$, because the curvature of the induced connection $\nabla^{E_H}$ 
is induced by the curvature of $\nabla^0$. Therefore, $\nabla^{E_H}$ produces a holomorphic 
structure on the principal $H$--bundle $E_H$ \cite[p.~9, Proposition 3.7]{Ko}. This holomorphic 
structure on $E_H$ is invariant, because $\nabla^{E_H}\, \in\, C(E_H)^G$.
\end{proof}

We will classify the space of all invariant holomorphic structures on a homogeneous principal 
$H$--bundle. For that purpose we need to consider invariant almost holomorphic structures.

An almost holomorphic structure on $E_H$ is a $C^\infty$ automorphism $J\, :\, TE_H\,\longrightarrow\, 
TE_H$ of vector bundles such that
\begin{itemize}
\item $J\circ J\, =\, -\text{Id}_{TE_H}$,

\item the projection $f$ in \eqref{prf} intertwines $J$ and the almost complex structure on 
$G/K$, meaning $f$ is an almost holomorphic map, and

\item the map $E_H\times H\, \longrightarrow\, E_H$ giving the action of $H$
on $E_H$ is almost holomorphic.
\end{itemize}
Note that each fiber of $f$ is identified with $H$
up to left-translations, and hence each fiber of $f$ has a complex structure given by the
complex structure of $H$. The above third condition implies that the complex structure
on any fiber of $f$ is the restriction of $J$.

So a holomorphic structure on the principal $H$--bundle $E_H$ is an integrable almost 
holomorphic structure on $E_H$. An almost holomorphic structure $J$ on $E_H$ will be called 
invariant if the action of $G$ on $E_H$ preserves the automorphism $J$. Note that an invariant 
holomorphic structure on $E_H$ is an integrable invariant almost holomorphic structure on $E_H$.

Giving an almost holomorphic structure on $E_H$ is equivalent to giving a complex distribution
$$D_J\, \subset\, TE_H\otimes{\mathbb C}$$ satisfying the following two conditions:
\begin{itemize}
\item the differential $$df\otimes {\mathbb C}\, :\, TE_H\otimes{\mathbb C}\, \longrightarrow\,
f^*T(G/K)\otimes{\mathbb C}$$ of the projection $f$ in \eqref{prf} maps $D_J$ isomorphically to
$T^{0,1}(G/K)\,\subset\, T(G/K)\otimes\mathbb C$, where $df$ is the differential of $f$, and

\item the distribution $D_J$ is preserved by the action of $H$ on $E_H$.
\end{itemize}
Note that the first condition implies that $D_J\bigcap\text{kernel}(df\otimes {\mathbb C})
\,=\, 0$ and $\dim D_J\,=\, \dim_{\mathbb C} G/K$.

Given a distribution $D_J$ satisfying the above two conditions, consider
$$D_J\oplus \text{kernel}(df\otimes {\mathbb C})^{0,1}\, \subset\, TE_H\otimes{\mathbb C}\, ,
$$
where $\text{kernel}(df\otimes {\mathbb C})\,=\, \text{kernel}(df\otimes {\mathbb C})^{1,0}
\oplus \text{kernel}(df\otimes {\mathbb C})^{0,1}$ is the type decomposition corresponding to
the complex structure on the fibers of $f$.
Then there is a unique almost complex structure on $E_H$ such that
the corresponding complex distribution $$T^{0,1}E_H\, \subset\, TE_H\otimes{\mathbb C}$$
is $D_J\oplus \text{kernel}(df\otimes {\mathbb C})^{0,1}$.

Let $\text{ad}(E_H)\, :=\, E_H\times^H {\mathfrak h}$ be the adjoint bundle associated to $E_H$ 
for the adjoint action of $H$ on its Lie algebra $\mathfrak h$. So sections of
$\text{ad}(E_H)$ over an open subset $U\, \subset\, G/K$ are identified with the $H$--invariant
sections of $\text{kernel}(df)$ over $f^{-1}(U)$.

Let $D_J\, \subset\, TE_H\otimes{\mathbb C}$ be a distribution giving an almost
complex structure on $E_H$. For any $C^\infty$ section $s$ of
$T^{0,1}(G/K)$ defined over an open subset $U\, \subset\, G/K$, let $\widehat{s}$ be
the unique $C^\infty$ section of $D_J$ over $f^{-1}(U)\, \subset\, E_H$ such that
$(df\otimes {\mathbb C}) (\widehat{s})\,=\, f^*s$. Then for any two $C^\infty$
sections $s$ and $t$ of $T^{0,1}(G/K)$ over $U$, the section
\begin{equation}\label{cc}
{\mathcal K}(D_J)(s,\, t)\, :=\, [\widehat{s},\, \widehat{t}]- \widehat{[s,\, t]}
\end{equation}
is an $H$--invariant section of $\text{kernel}(df)$ over $f^{-1}(U)$.
Therefore, ${\mathcal K}(D_J)(s,\, t)$ produces a section of $\text{ad}(E_H)\vert_U$. It
is straightforward to check that
${\mathcal K}(D_J)(\psi\cdot s,\, t)\,=\, \psi\cdot {\mathcal K}(D_J)(s,\, t)$ for
any locally defined $C^\infty$ function $\psi$ on $G/K$, and
${\mathcal K}(D_J)(s,\, t)\,=\, - {\mathcal K}(D_J)(t,\, s)$. Consequently,
${\mathcal K}(D_J)$ is a $C^\infty$ section of $\Omega^{0,2}_{G/K}\otimes \text{ad}(E_H)$.

Note that the above distribution $D_J$ is integrable if and only if ${\mathcal K}(D_J)\,=\, 0$.
The almost complex structure on $E_H$ given by $D_J$ is integrable if and only if
${\mathcal K}(D_J)\,=\, 0$.

The space of almost holomorphic structures on $E_H$ is an affine space for the vector space 
$C^\infty (G/K, \, \text{ad}(E_H)\otimes \Omega^{0,1}_{G/K})$. Note that from Theorem \ref{thm0} 
we know that the space of almost holomorphic structures on $E_H$ is nonempty. The actions of $G$ 
on $E_H$ and $G/K$ together produce an action of $G$ on $C^\infty (G/K, \, \text{ad}(E_H)\otimes 
\Omega^{0,1}_{G/K})$. The space of invariant almost holomorphic structures on $E_H$ is an affine 
space for the vector space $C^\infty (G/K, \, \text{ad}(E_H)\otimes\Omega^{0,1}_{G/K})^G$
of $G$--invariants in $C^\infty (G/K, \, \text{ad}(E_H)\otimes\Omega^{0,1}_{G/K})$; note that the 
space of invariant almost holomorphic structures on $E_H$ is nonempty, because the almost 
holomorphic structures on $E_H$ given by Theorem \ref{thm0} is invariant.

Since the translation action of $G$ on $G/K$ is transitive, the evaluation map
\begin{equation}\label{ep}
\epsilon\, :\, 
C^\infty (G/K, \, \text{ad}(E_H)\otimes \Omega^{0,1}_{G/K})^G\, \longrightarrow\, (\text{ad}(E_H)\otimes \Omega^{0,1}_{G/K})_{K/K}
\, ,\ \ s\, \longmapsto\, s(K/K)
\end{equation}
is injective.

Fix a point
\begin{equation}\label{z0}
z_0\, \in\, (E_H)_{K/K}\, ,
\end{equation}
so $(E_H,\, z_0)$ is a based homogeneous bundle. The map
$$H\, \longmapsto\, (E_H)_{K/K}\, , \ \ h\, \longmapsto\, z_0h
$$
identifies the fiber $(E_H)_{K/K}$ with $H$. Recall that
the fiber $\text{ad}(E_H)_{K/K}$ is a quotient of 
$(E_H)_{K/K}\times\mathfrak h$. The map $\mathfrak h\, \longrightarrow\,\text{ad}(E_H)_{K/K}$
that sends any $v\, \in\, \mathfrak h$ to the element of $\text{ad}(E_H)_{K/K}$
given by $(z_0,\, v)$, where $z_0$ is the element in \eqref{z0}, identifies $\mathfrak h$ 
with $\text{ad}(E_H)_{K/K}$. On the other hand, the two vector spaces $(T^{0,1}_{K/K}
(G/K))^*\,=\, (\Omega^{0,1}_{G/K})_{K/K}$ and
$T^{1,0}_{K/K}(G/K)$ are identified using the K\"ahler form on $G/K$, and hence
$(\Omega^{0,1}_{G/K})_{K/K}$ is identified with ${\mathfrak p}_+$ defined in \eqref{d1} (recall
that $T^{1,0}_{K/K}(G/K)$ is identified with ${\mathfrak p}_+$). So we have
\begin{equation}\label{z1}
(\Omega^{0,1}_{G/K})_{K/K}\,=\, {\mathfrak p}_+\, .
\end{equation}
Therefore, the map $\epsilon$ is \eqref{ep} is in fact an injective map 
\begin{equation}\label{ep2}
\epsilon\, :\, 
C^\infty (G/K, \, \text{ad}(E_H)\otimes \Omega^{0,1}_{G/K})^G\, \longrightarrow\,
{\mathfrak h}\otimes {\mathfrak p}_+\, .
\end{equation}

From Theorem \ref{thm0} we know that $E_H$ has a tautological invariant holomorphic structure. 
Since the space of almost holomorphic structures (respectively, invariant almost holomorphic 
structures) on $E_H$ is an affine space for $C^\infty (G/K, \, \text{ad}(E_H)\otimes 
\Omega^{0,1}_{G/K})$ (respectively, $C^\infty (G/K, \, \text{ad}(E_H)\otimes 
\Omega^{0,1}_{G/K})^G$), using this holomorphic structure on $E_H$ given by Theorem \ref{thm0} 
as the base point, the space of almost holomorphic structures (respectively, invariant almost 
holomorphic structures) on $E_H$ gets identified with $C^\infty (G/K, \, \text{ad}(E_H)\otimes 
\Omega^{0,1}_{G/K})$ (respectively, $C^\infty (G/K, \, \text{ad}(E_H)\otimes 
\Omega^{0,1}_{G/K})^G$).

The Lie algebra operation ${\mathfrak h}\otimes {\mathfrak h}\,\longrightarrow\,
{\mathfrak h}$ and the exterior multiplication
$$
{\mathfrak p}_+\otimes {\mathfrak p}_+ \, \longrightarrow\, \bigwedge\nolimits^2 {\mathfrak p}_+
$$
together define a homomorphism
\begin{equation}\label{mpl}
{\mathbf m}_+\, :\, ({\mathfrak h}\otimes {\mathfrak p}_+)^{\otimes 2}\,\longrightarrow
\, {\mathfrak h}\otimes \bigwedge\nolimits^2 {\mathfrak p}_+\, .
\end{equation}

\begin{proposition}\label{prop1}
Take an invariant almost holomorphic structure $$\beta\,\in\, C^\infty (G/K, \,
{\rm ad}(E_H)\otimes \Omega^{0,1}_{G/K})^G\, .$$ Then $\beta$ is integrable if and only if
$$
{\mathbf m}_+(\epsilon (\beta)\otimes \epsilon (\beta)) \,=\, 0\, ,
$$
where $\epsilon$ and ${\mathbf m}_+$ are constructed in \eqref{ep2} and \eqref{mpl}
respectively.
\end{proposition}

\begin{proof}
Let $\eta\, :\, K\, \longrightarrow\, H$ be the homomorphism constructed as in \eqref{eta} for 
the based homogeneous principal $H$--bundle $(E_H,\, z_0)$. As before, the center of $K$ is 
denoted by $Z_K$. Consider the action of $Z_K$ on $\mathfrak h$ obtained by combining $\eta$ 
with the adjoint action of $H$ on $\mathfrak h$; in other words, this action is given by the 
following composition of maps
$$
K\, \stackrel{\eta}{\longrightarrow}\, H \, \stackrel{\rm ad}{\longrightarrow}\,
\text{Aut}({\mathfrak h})\, .
$$
Let
\begin{equation}\label{ist}
{\mathfrak h}\,=\,\bigoplus_{\lambda\in (Z_K)^*} {\mathfrak h}_\lambda
\end{equation}
be the corresponding isotypical decomposition. Since $Z_K$ commutes with $K$, the action of $K$ on $\mathfrak h$,
constructed as above using $\eta$ and the adjoint action of $H$ on $\mathfrak h$, preserves the decomposition in \eqref{ist}.
Let
$${\mathbb G}({\mathfrak h}_\lambda)\, :=\, {\mathbb G}\times^K {\mathfrak h}_\lambda\, \longrightarrow\, G/K$$
be the vector bundle over $G/K$
associated to the principal $K$--bundle ${\mathbb G}$ (constructed in \eqref{f1})
for the $K$--module ${\mathfrak h}_\lambda$ in \eqref{ist}. From the decomposition in
\eqref{ist} we have the decomposition
\begin{equation}\label{ist2}
\text{ad}(E_H)\, =\, \bigoplus_{\lambda\in (Z_K)^*} {\mathbb G}({\mathfrak h}_\lambda)
\end{equation}
which is in fact a holomorphic decomposition.

Consider the isotypical decomposition
$$
{\mathfrak p}^{\mathbb C}\, :=\,
{\mathfrak p}\otimes_{\mathbb R} {\mathbb C}\,=\, {\mathfrak p}_+\oplus {\mathfrak p}_{-}
$$
in \eqref{d1}. Recall from \eqref{z1} that $(\Omega^{0,1}_{G/K})_{K/K}\,=\, {\mathfrak p}_+$.
Therefore, $Z_K$ acts on $(\Omega^{0,1}_{G/K})_{K/K}\,=\,
{\mathfrak p}_+$ through a single character. This character of $Z_K$, through which it
acts on ${\mathfrak p}_+$,
will be denoted by $\chi$. The character $\chi$ is actually nontrivial; indeed, this follows
from the fact that $G/K$ does not have any nonzero holomorphic one-form.

Take an invariant section $\beta\,\in\, C^\infty (G/K, \, {\rm ad}(E_H)\otimes \Omega^{0,1}_{G/K})^G$
as in the statement of the proposition. Therefore, the element $$\beta(K/K)\,\in\,
({\rm ad}(E_H)\otimes \Omega^{0,1}_{G/K})_{K/K}$$ is fixed by the action of $Z_K$
on $({\rm ad}(E_H)\otimes \Omega^{0,1}_{G/K})_{K/K}$ (in fact
it is fixed by $K$). Since $\beta(K/K)$ is fixed by the action of $Z_K$, it can be shown that
\begin{equation}\label{st}
\beta(K/K)\, \in\, ({\mathbb G}({\mathfrak h}_{\chi^{-1}})\otimes
\Omega^{0,1}_{G/K})_{K/K}\, \subset\, ({\rm ad}(E_H)\otimes \Omega^{0,1}_{G/K})_{K/K}\, ,
\end{equation}
where ${\mathbb G}({\mathfrak h}_{\chi^{-1}})$ is the direct summand in \eqref{ist2} for
the character $\chi^{-1}$ defined above. Indeed, this follows immediately from the
fact that $Z_K$ acts on the cotangent space $(\Omega^{0,1}_{G/K})_{K/K}$ as multiplication
by $\chi$. Since $\beta$ is $G$--invariant, and the action of $G$ on
$G/K$ is transitive, from \eqref{st} we conclude that
\begin{equation}\label{st2}
\beta\,\in\, C^\infty(G/K,\, {\mathbb G}({\mathfrak h}_{\chi^{-1}})\otimes
\Omega^{0,1}_{G/K})^G\, .
\end{equation}

Let $\overline{\partial}^0_{\text{ad}(E_H)}\, :\, \text{ad}(E_H)\, \longrightarrow\,
\text{ad}(E_H)\otimes \Omega^{0,1}_{G/K}$ be the Dolbeault operator on $\text{ad}(E_H)$
induced by the tautological holomorphic structure on $E_H$ (see Theorem \ref{thm0}).
Since the decomposition in \eqref{ist2} is holomorphic, from \eqref{st2} it follows
that
$$
\overline{\partial}^0_{\text{ad}(E_H)}(\beta) \,\in\,
C^\infty(G/K,\, {\mathbb G}({\mathfrak h}_{\chi^{-1}})\otimes \Omega^{0,2}_{G/K})^G\, .
$$
So $\overline{\partial}^0_{\text{ad}(E_H)}(\beta)(K/K)$ is a $K$--invariant element
of ${\mathfrak h}_{\chi^{-1}}\otimes \bigwedge^2 {\mathfrak p}_+$. But $Z_K$ acts on
${\mathfrak h}_{\chi^{-1}}\otimes \bigwedge^2 {\mathfrak p}_+$ as multiplication via the
nontrivial character $\chi^{-1}\cdot \chi^2 \,=\, \chi$. Hence we conclude that
\begin{equation}\label{bz}
\overline{\partial}^0_{\text{ad}(E_H)}(\beta)\,=\, 0\, .
\end{equation}

Let $$D_\beta \, \subset\, TE_H\otimes{\mathbb C}$$ be the distribution corresponding to
the almost complex structure $\beta$ on $E_H$. Let
$$D_0 \, \subset\, TE_H\otimes{\mathbb C}$$ be the distribution corresponding to
the tautological almost complex structure on $E_H$ (see Theorem \ref{thm0}); note that
the tautological almost complex structure on $E_H$ corresponds to the identically zero
section of $\text{ad}(E_H)\otimes \Omega^{0,1}_{G/K}$. We have
\begin{equation}\label{es}
{\mathcal K}(D_\beta)\,=\, {\mathcal K}(D_0) +
\overline{\partial}^0_{\text{ad}(E_H)}(\beta) + {\mathbf m}_+
(\epsilon (\beta)\otimes \epsilon (\beta))\, ,
\end{equation}
where ${\mathcal K}(D_\beta)$ and ${\mathcal K}(D_0)$ are constructed as in \eqref{cc}
for $D_\beta$ and $D_0$ respectively.
Now $${\mathcal K}(D_0)\,=\, 0$$ because the tautological almost holomorphic
structure on $E_H$ is integrable. Hence using \eqref{bz}, from \eqref{es} we conclude that
${\mathcal K}(D_\beta)\,=\, 0$ if and only if
${\mathbf m}_+(\epsilon (\beta)\otimes \epsilon (\beta))\,=\, 0$. Since the almost complex
structure on $E_H$ corresponding to $\beta$ is integrable if and only if
${\mathcal K}(D_\beta)\,=\,0$ (see \cite[p.~9, Proposition 3.7]{Ko}), the
proposition follows.
\end{proof}

\begin{theorem}\label{thm1}
There is a natural bijection between the following two:
\begin{enumerate}
\item Isomorphism classes of based homogeneous principal $H$--bundles with an invariant 
holomorphic structure.

\item Pairs of the form $(\eta,\, \beta)$, where $\eta\, :\, K\, 
\longrightarrow\, H$ is a homomorphism and $\beta\, \in\, ({\mathfrak h}\otimes {\mathfrak 
p}_+)^K$ with ${\mathbf m}_+(\beta\otimes\beta)\,=\, 0$, where ${\mathbf m}_+$ is defined
in \eqref{mpl}.
\end{enumerate}
\end{theorem}

\begin{proof}
Let $(E_H, \,\rho)$ be a homogeneous holomorphic principal $H$--bundle with a base point
$z_0\,\in\, (E_H)_{K/K}$. For any $k\, \in\, K$, let $\eta(k)\, \in\, H$ be the unique
element that satisfies the equation
$$
\rho(k,\, z_0)\,=\, z_0\eta(k)\, .
$$
It was shown in the proof of Lemma \ref{lem1} that $\eta$ is a homomorphism of groups.

Recall that the space of all $G$--invariant almost complex structures on the
underlying $C^\infty$ homogeneous principal $H$--bundle $E_H$ is
an affine space for vector space $$C^\infty (G/K, \, {\rm ad}(E_H)\otimes \Omega^{0,1}_{G/K})^G$$
of $G$--invariants in $C^\infty (G/K, \, {\rm ad}(E_H)\otimes \Omega^{0,1}_{G/K})$. Therefore,
using the tautological holomorphic structure on $E_H$ given by Theorem \ref{thm0},
the space of all $G$--invariant almost complex structures on the
$C^\infty$ homogeneous principal $H$--bundle $E_H$ gets identified with
$C^\infty (G/K, \, {\rm ad}(E_H)\otimes \Omega^{0,1}_{G/K})^G$. Now let
$$
\widetilde{\beta}\, \in\, C^\infty (G/K, \, {\rm ad}(E_H)\otimes \Omega^{0,1}_{G/K})^G
$$
be the element corresponding to the given holomorphic structure on $E_H$.

The fiber $({\rm ad}(E_H)\otimes \Omega^{0,1}_{G/K})_{K/K}$ is identified with
${\mathfrak h}\otimes {\mathfrak p}_+$ using the base point $z_0$. Hence
$$
\beta\, :=\, \widetilde{\beta}(K/K)\, \in\, {\mathfrak h}\otimes {\mathfrak p}_+\, .
$$
{}From Proposition \ref{prop1} it follows that ${\mathbf m}_+(\beta\otimes\beta)\,=\, 0$.

To prove the converse, take any pair $(\eta,\, \beta)$, where $\eta\, :\, K\,
\longrightarrow\, H$ is a homomorphism and $\beta\, \in\, ({\mathfrak h}\otimes {\mathfrak
p}_+)^K$ with ${\mathbf m}_+(\beta\otimes\beta)\,=\, 0$. Let $E_H$ be the homogeneous
$C^\infty$ principal $H$--bundle over $G/K$ given by $\eta$ using Lemma \ref{lem1}.

Since $({\rm ad}(E_H)\otimes \Omega^{0,1}_{G/K})_{K/K}\,=\, {\mathfrak h}\otimes {\mathfrak p}_+$
and $\beta\, \in\, ({\mathfrak h}\otimes {\mathfrak p}_+)^K$, there is a unique invariant section
$$
\widetilde{\beta}\, \in\, C^\infty (G/K, \, {\rm ad}(E_H)\otimes \Omega^{0,1}_{G/K})^G
$$
such that $\widetilde{\beta}(K/K)\,=\, \beta$. As mentioned above, $C^\infty (G/K, \, {\rm 
ad}(E_H)\otimes \Omega^{0,1}_{G/K})^G$ is identified with the space of all $G$--invariant
almost complex structures on the $C^\infty$ homogeneous principal $H$--bundle $E_H$. Equip
$E_H$ with the $G$--invariant almost complex structure corresponding to
the above section $\widetilde{\beta}$.
Since ${\mathbf m}_+(\beta\otimes\beta)\,=\, 0$, from Proposition \ref{prop1} it follows that
this almost complex structure is integrable.
\end{proof}

Take any two pairs $(\eta,\, \beta)$ and $(\eta',\, \beta')$ as in Theorem \ref{thm1}. They
will be called {\it equivalent} if there is an element $h\, \in\, H$ such that
\begin{itemize}
\item $\eta'(g)\, =\, h^{-1}\eta(g) h$ for all $g\, \in\, K$, and

\item $\beta'\,=\, (\text{ad}(h)\otimes {\rm Id}_{{\mathfrak p}_+})(\beta)$.
\end{itemize}

\begin{corollary}\label{cor1}
There is a natural bijection between the following two:
\begin{enumerate}
\item Isomorphism classes of principal $H$--bundles with an invariant holomorphic structure.

\item Equivalence classes of pairs of the form $(\eta,\, \beta)$, where $\eta\, :\, K\, 
\longrightarrow\, H$ is a homomorphism and $\beta\, \in\, ({\mathfrak h}\otimes {\mathfrak 
p}_+)^K$ with ${\mathbf m}_+(\beta\otimes\beta)\,=\, 0$.
\end{enumerate}
\end{corollary}

\begin{proof}
Let $(E_H, \,\rho)$ be a homogeneous holomorphic principal $H$--bundle with a base point
$z_0\,\in\, (E_H)_{K/K}$. As in the proof of Theorem \ref{thm1}, for any $k\, \in\, K$,
let $\eta(k)\, \in\, H$ be the unique element that satisfies the equation
$$
\rho(k,\, z_0)\,=\, z_0\eta(k)\, .
$$
Now take any $h_0\, \in\, H$ and set $z_0h$ to be the new base point on $E_H$. Let
$\eta'\, :\, K\, \longrightarrow\, H$ be the homomorphism corresponding to
the base point $z_0h$, so $\rho(k,\, z_0h)\,=\, z_0h\eta'(k)$ for all $k\, \in\, K$.
Now, for any $k\, \in\, K$, we have
$$
z_0h\eta'(k)\,=\, \rho(k,\, z_0h) \,=\, \rho(k,\, z_0)h\,=\, z_0\eta(k)h\, .
$$
This implies that $\eta'(k)\,=\, h^{-1}\eta(k)h$. Now it is straightforward to deduce
the corollary from Theorem \ref{thm1}.
\end{proof}

\section{Homogeneous co-Higgs bundles}

We treat co-Higgs bundles first and then ordinary Higgs bundles in the next section.

Let $(E_H,\, \rho)$ be a homogeneous holomorphic principal $H$--bundle over $G/K$. The action 
of $G$ on $E_H$ induces an action of $G$ on $\text{ad}(E_H)$. The actions of $G$ on $G/K$ and 
$\text{ad}(E_H)$ together produce an action of $G$ on the holomorphic vector bundle 
$\text{ad}(E_H)\otimes T^{1,0} (G/K)$. Take a holomorphic section $$\theta\, \in\, H^0(G/K,\, 
\text{ad}(E_H)\otimes T^{1,0}(G/K))\, .$$ Using the Lie algebra structure of the fibers of 
$\text{ad}(E_H)$, we have
$$
\theta\bigwedge\theta\, \in\, H^0(G/K,\, \text{ad}(E_H)\otimes \bigwedge\nolimits^2
T^{1,0}(G/K))\, .
$$

An \textit{invariant co-Higgs} field on $E_H$ is a
holomorphic section $$\theta\, \in\, H^0(G/K,\, \text{ad}(E_H)\otimes T^{1,0}(G/K))$$ such that
\begin{enumerate}
\item $\theta\bigwedge\theta\,=\, 0$, and

\item the action of $G$ on $\text{ad}(E_H)\otimes T^{1,0}(G/K)$ fixes the section $\theta$.
\end{enumerate}

A \textit{homogeneous} co-Higgs $H$--bundle is a homogeneous holomorphic principal $H$--bundle 
equipped with an invariant co-Higgs field. Two homogeneous co-Higgs bundles
$(E'_H,\, \rho',\, \theta')$ and $(E''_H,\, \rho'',\, \theta'')$ are \textit{isomorphic}
if there is a holomorphic isomorphism of principal $H$--bundles
$\alpha\, :\, E'_H\, \longrightarrow\, E''_H$
that satisfies the following two conditions:
\begin{itemize}
\item $\alpha$ intertwines the actions of $G$ on $E'_H$ and $E''_H$, and

\item the isomorphism $\text{ad}(E'_H)\otimes T^{1,0}(G/K)\, \longrightarrow\, \text{ad}(E''_H)
\otimes T^{1,0}(G/K)$ constructed using $\alpha$ takes the section $\theta'$ to $\theta''$.
\end{itemize}

A based homogeneous co-Higgs $H$--bundle on $G/K$ is a homogeneous co-Higgs $H$--bundle $(E'_H,\, 
\rho',\, \theta')$ equipped with a base point $z'\, \in\, (E'_H)_{K/K}$ over
$K/K\, \in\, G/K$. Two based 
homogeneous co-Higgs bundles
$$
(E'_H,\, \rho',\, \theta',\, z')\ \ \text{ and }\ \ (E''_H,\, \rho'',\, 
\theta'',\, z'')
$$
are \textit{isomorphic} if there is an isomorphism between $(E'_H,\, \rho',\, \theta')$ and 
$(E''_H,\, \rho'',\, \theta'')$ that takes $z'$ to $z''$.

Using the Lie algebra operation on $\mathfrak h$ we define the homomorphism
\begin{equation}\label{bm}
{\mathbf m}\, :\, ({\mathfrak h}\otimes {\mathfrak p}_+)\otimes
({\mathfrak h}\otimes {\mathfrak p}_+)\,\longrightarrow
\, {\mathfrak h}\otimes{\mathfrak p}_+\otimes{\mathfrak p}_+\, .
\end{equation}

\begin{theorem}\label{thm2}
There is a natural bijection between the following two:
\begin{enumerate}
\item Isomorphism classes of based homogeneous principal co-Higgs $H$--bundles on $G/K$.

\item Triples of the form $(\eta,\, \beta,\, \varphi)$, where
$\eta\, :\, K\, \longrightarrow\, H$ is a homomorphism and
$$\beta, \, \varphi \, \in\, ({\mathfrak h}\otimes {\mathfrak p}_+)^K$$
such that
$$
{\mathbf m}_+(\beta\otimes\beta)\,=\, {\mathbf m}_+(\varphi\otimes\varphi)\,=\,0\,=\,
{\mathbf m}(\beta\otimes\varphi)\, ,
$$
where ${\mathbf m}_+$ and ${\mathbf m}$ are constructed in \eqref{mpl} and \eqref{bm}
respectively.
\end{enumerate}
\end{theorem}

\begin{proof}
Let $(E_H,\, \rho,\, \theta,\, z_0)$ be a based homogeneous co-Higgs $H$--bundle.
From Theorem \ref{thm1} we know that the based homogeneous holomorphic principal
$H$--bundle $(E_H,\, \rho,\, z_0)$ gives a pair $(\eta,\, \beta)$, where $\eta\, :\, K\,
\longrightarrow\, H$ is a homomorphism and $\beta\, \in\, ({\mathfrak h}\otimes {\mathfrak
p}_+)^K$ with ${\mathbf m}_+(\beta\otimes\beta)\,=\, 0$.

Consider
$$
\varphi\,=\, \theta(K/K)\,\in\, (\text{ad}(E_H)\otimes T^{1,0}(G/K))_{K/K}
\,=\, {\mathfrak h}\otimes{\mathfrak p}_+\, ;
$$
as before, the fiber ${\rm ad}(E_H)_{K/K}$ is identified with $\mathfrak h$ using $z_0$,
while the identification between $T^{1,0}_{K/K} (G/K)$ and ${\mathfrak p}_+$ is the one
in Section \ref{se2.2}. Since the section $\theta$ is $G$--invariant, it follow that
$\varphi\, \in\,
({\mathfrak h}\otimes {\mathfrak p}_+)^K$. The condition that $\theta\bigwedge\theta(K/K)\,=\, 0$
is equivalent to the condition
that ${\mathbf m}_+(\varphi\otimes\varphi)\,=\,0$.

Next it will be shown that $${\mathbf m}(\beta\otimes\varphi)\,=\, 0\, .$$
For that, first note the Dolbeault operator $\overline{\partial}'$ for the holomorphic
vector bundle $\text{ad}(E_H)\otimes T^{1,0} (G/K)$ satisfies the equation
\begin{equation}\label{d0}
\overline{\partial}'\,=\, \overline{\partial}'_0+\beta\, ,
\end{equation}
where $\overline{\partial}'_0$ denotes the Dolbeault operator on
$\text{ad}(E_H)\otimes T^{1,0} (G/K)$ corresponding to the tautological connection on the
homogeneous principal $H$--bundle $E_H$ obtained in Theorem \ref{thm0}. Now the given condition
that the section $\theta$ is holomorphic implies that $\overline{\partial}'(\theta)\,=\, 0$, and
hence from \eqref{d0} we have
\begin{equation}\label{a1}
(\overline{\partial}'_0+\beta)(\theta)\,=\, 0\, .
\end{equation}
On the other hand, we have 
\begin{equation}\label{a2}
\overline{\partial}'_0(\theta)\,=\, 0
\end{equation}
because the section $\theta$ is invariant; recall that the tautological holomorphic structure
has the property that any invariant section is holomorphic. Now combining \eqref{a1} and
\eqref{a2} we conclude that ${\mathbf m}(\beta\otimes\varphi)\,=\, 0$.

To prove the converse, take a triple $(\eta,\, \beta,\, \varphi)$ satisfying the conditions
in the statement of the theorem. From Theorem \ref{thm1} we know that the pair
$(\eta,\, \beta)$ gives a holomorphic homogeneous principal $H$--bundle $(E_H,\, \rho)$.
Since $\varphi$ is $K$--invariant, there is a unique $G$--invariant $C^\infty$ section
$$
\theta\, \in\, C^\infty(G/K,\, \text{ad}(E_H)\otimes T^{1,0}(G/K))^G
$$
such that $\theta(K/K)\,=\, \varphi$.

The evaluation $$\theta\bigwedge\theta(K/K)\, \in\,
(\text{ad}(E_H)\otimes \bigwedge\nolimits^2 T^{1,0}(G/K))_{K/K}\,=\,
{\mathfrak h}\otimes \bigwedge\nolimits^2 {\mathfrak p}_+$$ coincides with
${\mathbf m}_+(\varphi\otimes\varphi)$. Since $\varphi$ is $G$--equivariant, from the
given condition that ${\mathbf m}_+(\varphi\otimes\varphi)\,=\,0$
we conclude that $\theta\bigwedge\theta\,=\, 0$.

Next consider
$$
\overline{\partial}'(\theta)\, \in\, C^\infty(G/K,\, \text{ad}(E_H)\otimes T^{1,0}(G/K)
\otimes{\Omega}^{0,1}_{G/K})\, ,
$$
where $\overline{\partial}'$ is the Dolbeault operator corresponding to the
holomorphic structure on the vector bundle $\text{ad}(E_H)\otimes T^{1,0} (G/K)$.
This section $\overline{\partial}'(\theta)$ is $G$--invariant, because $\theta$ is
$G$--invariant, and the holomorphic structure is preserved by the action of $G$. Let
$$
\overline{\partial}'(\theta)(K/K)\,\in\, (\text{ad}(E_H)\otimes T^{1,0}(G/K)\otimes
{\Omega}^{0,1}_{G/K})_{K/K}\,=\, {\mathfrak h}\otimes{\mathfrak p}_+\otimes{\mathfrak p}_+
$$
be the evaluation of this section at the point $K/K\, \in\, G/K$. Combining \eqref{d0}
with the fact that $\overline{\partial}'_0(\theta)\,=\, 0$ it follows that
\begin{equation}\label{f12}
\overline{\partial}(\theta)(K/K)\,=\, {\mathbf m}(\beta\otimes\varphi)\, .
\end{equation}
Since $\overline{\partial}(\theta)$ is $G$--invariant, and ${\mathbf m}(\beta\otimes\varphi)
\,=\, 0$, from \eqref{f12} we conclude that $\overline{\partial}(\theta)\,=\, 0$. In other words,
the section $\theta$ is holomorphic.
\end{proof}

Take any two triples $(\eta,\, \beta,\, \varphi)$ and $(\eta',\, \beta',\,
\varphi')$ satisfying the conditions in Theorem \ref{thm2}. They
will be called {\it equivalent} if there is an element $h\, \in\, H$ such that
\begin{itemize}
\item $\eta'(g)\, =\, h^{-1}\eta(g) h$ for all $g\, \in\, K$,

\item $\beta'\,=\, (\text{ad}(h)\otimes {\rm Id}_{{\mathfrak p}_+})(\beta)$, and.

\item $\varphi'\,=\, (\text{ad}(h)\otimes {\rm Id}_{{\mathfrak p}_+})(\varphi)$.
\end{itemize}

The following analogue of Corollary \ref{cor1} is a straightforward consequence
of Theorem \ref{thm2}.

\begin{corollary}\label{cor2}
There is a natural bijection between the following two:
\begin{enumerate}
\item Isomorphism classes of homogeneous co-Higgs $H$--bundles.

\item Equivalence classes of triples $(\eta,\, \beta,\, \varphi)$, where
$\eta\, :\, K\, \longrightarrow\, H$ is a homomorphism and
$$\beta, \, \varphi \, \in\, ({\mathfrak h}\otimes {\mathfrak p}_+)^K$$
such that
$$
{\mathbf m}_+(\beta\otimes\beta)\,=\, {\mathbf m}_+(\varphi\otimes\varphi)\,=\,0\,=\,
{\mathbf m}(\beta\otimes\varphi)\, ,
$$
where ${\mathbf m}_+$ and ${\mathbf m}$ are constructed in \eqref{mpl} and \eqref{bm}
respectively.
\end{enumerate}
\end{corollary}

\section{Homogeneous Higgs bundles}

As before, $(E_H,\, \rho)$ is a homogeneous holomorphic principal $H$--bundle over $G/K$. The 
actions of $G$ on $G/K$ and $\text{ad}(E_H)$ together produce an action of $G$ on the 
holomorphic vector bundle $\text{ad}(E_H)\otimes \Omega^{1,0}_{G/K}$. For any $\theta\, \in\, 
H^0(G/K,\, \text{ad}(E_H)\otimes \Omega^{1,0}_{G/K})$, we have
$$
\theta\bigwedge\theta\, \in\, H^0(G/K,\, \text{ad}(E_H)\otimes \bigwedge\nolimits^2
\Omega^{1,0}_{G/K})\, ,
$$
which is defined using the Lie algebra structure of the fibers of $\text{ad}(E_H)$.
In analogy to the preceding section, an \textit{invariant Higgs} field on $E_H$ is a
holomorphic section $$\theta\, \in\, H^0(G/K,\, \text{ad}(E_H)\otimes \Omega^{1,0}_{G/K})$$
such that
\begin{enumerate}
\item $\theta\bigwedge\theta\,=\, 0$, and

\item the action of $G$ on $\text{ad}(E_H)\otimes\Omega^{1,0}_{G/K}$ fixes the section $\theta$.
\end{enumerate}

Then, a \textit{homogeneous} Higgs $H$--bundle is a homogeneous holomorphic principal $H$--bundle 
equipped with an invariant Higgs field. \emph{Based} homogeneous Higgs $H$--bundles on
$G/K$ and isomorphisms between them are defined in exactly the same way as for the co-Higgs case.

Now, the Lie algebra operation ${\mathfrak h}\otimes {\mathfrak h}\,\longrightarrow\,
{\mathfrak h}$ and the exterior multiplication
$$
{\mathfrak p}_-\otimes {\mathfrak p}_- \, \longrightarrow\, \bigwedge\nolimits^2 {\mathfrak p}_-
$$
together define a homomorphism
\begin{equation}\label{mpl2}
{\mathbf m}_-\, :\, ({\mathfrak h}\otimes {\mathfrak p}_-)^{\otimes 2}\,\longrightarrow
\, {\mathfrak h}\otimes \bigwedge\nolimits^2 {\mathfrak p}_-\, .
\end{equation}

\begin{theorem}\label{prop2}
There is a natural bijection between the following two:
\begin{enumerate}
\item Isomorphism classes of based homogeneous principal Higgs $H$--bundles on $G/K$.

\item Triples of the form $(\eta,\, \beta,\, \varphi)$, where
$\eta\, :\, K\, \longrightarrow\, H$ is a homomorphism,
$$\beta\, \in\, ({\mathfrak h}\otimes {\mathfrak p}_+)^K
\ \ and \ \ \varphi \, \in\, ({\mathfrak h}\otimes {\mathfrak p}_-)^K$$
such that
$$
{\mathbf m}_+(\beta\otimes\beta)\,=\, 0\,=\, {\mathbf m}_-(\varphi\otimes\varphi)\,=\,
{\mathbf m}(\beta\otimes\varphi)\, ,
$$
where ${\mathbf m}_+$, ${\mathbf m}_-$ and ${\mathbf m}$ are constructed in \eqref{mpl},
\eqref{mpl2} and \eqref{bm} respectively.
\end{enumerate}
\end{theorem}

\begin{proof}
Consider the fiber $(\Omega^{1,0}_{G/K})_{K/K}$ of $(\Omega^{1,0}_{G/K})$ over
$K/K\,\in\, G/K$. We note that
$$
(\Omega^{1,0}_{G/K})_{K/K}\,=\, (T^{1,0}_{K/K}(G/K))^*\,=\, {\mathfrak p}^*_+\,=\,
{\mathfrak p}_-\, .
$$
Now the proof of the theorem is similar to that of Theorem \ref{thm2}. We omit the details.
\end{proof}

Take two triples $(\eta,\, \beta,\, \varphi)$ and $(\eta',\, \beta',\, \varphi')$, where
$\eta,\, \eta'\, :\, K\, \longrightarrow\, H$ are homomorphisms,
$$\beta\, \in\, ({\mathfrak h}\otimes {\mathfrak p}_+)^K
\ \ and \ \ \varphi \, \in\, ({\mathfrak h}\otimes {\mathfrak p}_-)^K$$
with respect to $\eta$
and
$$\beta'\, \in\, ({\mathfrak h}\otimes {\mathfrak p}_+)^K
\ \ and \ \ \varphi' \, \in\, ({\mathfrak h}\otimes {\mathfrak p}_-)^K$$
with respect to $\eta'$, such that
$$
{\mathbf m}_+(\beta\otimes\beta)\,=\, 0\,=\, {\mathbf m}_-(\varphi\otimes\varphi)\,=\,
{\mathbf m}(\beta\otimes\varphi)
$$
and
$$
{\mathbf m}_+(\beta'\otimes\beta')\,=\, 0\,=\, {\mathbf m}_-(\varphi'\otimes\varphi')\,=\,
{\mathbf m}(\beta'\otimes\varphi')\, ,
$$
where ${\mathbf m}_+$, ${\mathbf m}_-$ and ${\mathbf m}$ are constructed in \eqref{mpl},
\eqref{mpl2} and \eqref{bm} respectively. We will say that
$(\eta,\, \beta,\, \varphi)$ and $(\eta',\, \beta',\, \varphi')$ are \textit{equivalent}
if there is an element $h\, \in\, H$ such that
\begin{itemize}
\item $\eta'(g)\, =\, h^{-1}\eta(g) h$ for all $g\, \in\, K$,

\item $\beta'\,=\, (\text{ad}(h)\otimes {\rm Id}_{{\mathfrak p}_+})(\beta)$, and.

\item $\varphi'\,=\, (\text{ad}(h)\otimes {\rm Id}_{{\mathfrak p}_-})(\varphi)$.
\end{itemize}

Note that the above equivalence relation is very similar to the equivalence relation used in
Corollary \ref{cor2}.

Theorem \ref{prop2} has the following analog of Corollary \ref{cor2}.

\begin{corollary}\label{cor3}
There is a natural bijection between the following two:
\begin{enumerate}
\item Isomorphism classes of homogeneous principal Higgs $H$--bundles on $G/K$.

\item Equivalence classes of triples $(\eta,\, \beta,\, \varphi)$, where
$\eta\, :\, K\, \longrightarrow\, H$ is a homomorphism,
$$\beta\, \in\, ({\mathfrak h}\otimes {\mathfrak p}_+)^K
\ \ and \ \ \varphi \, \in\, ({\mathfrak h}\otimes {\mathfrak p}_-)^K$$
such that
$$
{\mathbf m}_+(\beta\otimes\beta)\,=\, 0\,=\, {\mathbf m}_-(\varphi\otimes\varphi)\,=\,
{\mathbf m}(\beta\otimes\varphi)\, ,
$$
where ${\mathbf m}_+$, ${\mathbf m}_-$ and ${\mathbf m}$ are constructed in \eqref{mpl},
\eqref{mpl2} and \eqref{bm} respectively.
\end{enumerate}
\end{corollary}

\section{Remark on moduli spaces}

The classifications above do not distinguish between (semi)stable and unstable objects in the 
sense of geometric invariant theory. That being said, Theorem \ref{prop2}
and Theorem \ref{thm2} endow the sets of isomorphism classes of based homogeneous principal
Higgs and co-Higgs 
$H$-bundles on $G/K$ with the structures of algebraic varieties. In the Higgs case, this space 
is the subspace of
$$
{\rm Hom}(K,\, H)\times(\mathfrak h\otimes\mathfrak p_+)^K\times(\mathfrak 
h\otimes\mathfrak p_-)^K
$$
consisting of all $(\eta,\, \beta,\, \varphi)$ satisfying the algebraic equations
$$
{\mathbf m}_+(\beta\otimes\beta)\,=\, 0\,=\, {\mathbf m}_-(\varphi\otimes\varphi)\,=\,
{\mathbf m}(\beta\otimes\varphi)\, .
$$
The space of isomorphism classes of homogeneous based principal
co-Higgs $H$-bundles on $G/K$ is the subspace of
$$
\mbox{Hom}(K,\, H)\times(\mathfrak h\otimes\mathfrak p_+)^K
\times ({\mathfrak h}\otimes {\mathfrak p}_+)^K
$$
defined by the locus of solutions $(\eta,\, \beta,\, \varphi)$ of the algebraic equations
$$
{\mathbf m}_+(\beta\otimes\beta)\,=\, {\mathbf m}_+(\varphi\otimes\varphi)\,=\,0\,=\,
{\mathbf m}(\beta\otimes\varphi)\, .
$$
In either case, the subspace can be regarded as a moduli space for the associated 
classification problem.

Under the assumption that the group $H$ is an affine algebraic group defined over $\mathbb C$, it is possible to construct from
Corollary \ref{cor2} (respectively, Corollary \ref{cor3}) a geometric-invariant-theoretic moduli space of
isomorphism classes of homogeneous principal co-Higgs (respectively, Higgs) bundles.

\section*{Acknowledgements}

The authors thank the anonymous referee for helpful comments and acknowledge Arghya Mondal for useful discussions.  The first-named author thanks
the Centre de recherches math\'ematiques (CRM), Montr\'eal, for hospitality while
this work was carried out. He is partially supported
by a J. C. Bose Fellowship.  The second-named author is partially supported by an NSERC Discovery Grant and by the Canadian Tri-Agency New Frontiers in Research Fund (Exploration Stream).


\end{document}